\documentclass[12pt,oneside,english]{amsart}
\usepackage[latin1]{inputenc}
\usepackage{amssymb}

\makeatletter
\newtheorem{theorem}{Theorem}
\newtheorem{lemma}{Lemma}
\newtheorem{proposition}{Proposition}

\newtheorem{remark}{Remark}

\usepackage{babel}

\makeatother
\begin{document}
\baselineskip=17pt

\title{Representation of positive integers by the form $x^3+y^3+z^3-3xyz$}

\author{Vladimir Shevelev}
\address{Departments of Mathematics \\Ben-Gurion University of the
 Negev\\Beer-Sheva 84105, Israel. e-mail:shevelev@bgu.ac.il}


\begin{abstract}
We study a representation of positive integers by the form $x^3+y^3+z^3-3xyz$ in the
conditions $0\leq x\leq y\leq z,\enskip z\geq x+1.$
\end{abstract}

\maketitle
\section{Introduction}
Let $F(x,y,z)=x^3+y^3+z^3-3xyz.$ For a positive integer $n,$ denote by $\nu(n)$
the number of ways to write $n$ in the form $F(x,y,z)$ in the
conditions $0\leq x\leq y\leq z,\enskip z\geq x+1.$ Indeed, the case $z=x$
is not interesting, since in this case $F(x,y,z)=F(x,x,x)=0.$ Below we proved the following
results:\newline
(i)for every positive $n,$ except for $n\equiv\pm3 \pmod9$ (cf.A074232 \cite{3}),
 $\nu(n)>=1;$\newline
(ii) for the exceptional $n,$ $\nu(n)=0;$\newline
(iii) for every prime $p\neq3,$ $\nu(p)=\nu(2p)=1;$\newline
(iv)\enskip$\limsup (\nu(n))=\infty;$\newline
(v) for every positive $n,$ there exists $k$ such that $\nu(k)=n.$
\section{Lower estimate of $F(x,y,z)$}
\begin{proposition}\label{prop1}
If $z\geq x+1,$ then
\begin{equation}\label{1}
F(x,y,z)\geq3z-2.
\end{equation}
\end{proposition}
\begin{proof}
Previously note that\newline
a) $F((z-1),(z+1),(z+1))=12z+4>3(z+1)-2;$\newline
b) $F(z,(z+1),(z+1))\enskip=\enskip 3z+2>3(z+1)-2;\newline
c)\enskip $F(z,\enskip z,\enskip(z+1))\enskip =\enskip 3(z+1)\enskip-\enskip2.\newline
Now we use induction over $z\geq1.$ Evidently, for $z=1,$ when either $(x,y)=(0,0)$
or $(x,y)=(0,1),$ the inequality (\ref{1}) holds. Suppose (\ref{1}) holds for some
$z\geq1.$ Now setting $z:=z+1,$ in view of a),\enskip b),\enskip c), we can take $0\leq x\leq z-1,
\enskip y\leq z.$ Then $F(x,y,z+1)=F(x,y,z)+3z^2+3z+1-3xy$ and, according the supposition,
$F(x,y,z+1)\geq (3z-2)+3z^2+3z+1-3(z-1)z=9z-1>3(z+1)-2.$
\end{proof}
The second proof.
\begin{proof}
We have
$$ F'_z(x,y,z)=3z^2-3xy\geq3z^2-3(z-1)z=3z\geq3.$$
So, for any fixed $x,y,$ $F(x,y,z)$ increases over $z.$ Hence, $F(x,y,z)\geq
F(x,y,z_{min}).$ 1) In the case $y=x+1,$ $z_{min}=x+1;$ 2) If $y<x+1,$ then $y=x.$ Since
 $z>x$, then $z_{min}=x+1;$ 3) If $y>x+1,$ then $z_{min}=y.$ \newline
 In case 1) $F(x,y,z)\geq F(x,x+1,x+1)=3x+2=3(z-1)+2=3z-1>3z-2;$\newline
 In case 2) $F(x,y,z)\geq F(x,x,x+1)=3x+1=3z-2;$\newline
 In case 3) $F(x,y,z)\geq F(x,y,y)=x^3+2y^3-3xy^2.$ Note that
 $F(x,y,y)'_y=6y^2-6xy\geq6y^2-6(y-2)y=12y\geq24.$ Since $y_{min}=x+2,$ then we have
 $F(x,y,z)\geq F(x,y,y)\geq F(x,x+2,x+2)=12x+16=12(z-2)+16\geq3z-2.$
 \end{proof}

 \begin{proposition}\label{prop2}
If $z\geq x+2,$ then
\begin{equation}\label{2}
F(x,y,z)\geq9z-10.
\end{equation}

\end{proposition}

Here there exist also at least two possibilities of proof. We show the second way.
\begin{proof}
Again
$$ F'_z(x,y,z)=3z^2-3xy\geq3z^2-3(z-2)z=6z\geq12.$$
1)-3) $y=x,x+1,x+2$ respectively, $z_{min}=x+2;$\newline
4) $y>x+2, z_{min}=y.$\newline
 We have\newline
 in case 1) $F(x,x,x+2)=12x+8=12z-16\geq9z-10,\enskip z\geq2;$\newline
 in case 2) $F(x,x+1,x+2)=9x+9=9(z-1)>9z-10;$\newline
 in case 3) $F(x,x+2,x+2)=12x+16=12z-8>9z-10;$\newline
 in case 4) $F(x,y,y)=x^3+2y^3-3xy^2.$ As in proof of Proposition \ref{prop1},
 $F(x,y,y)'_y>0.$ Since $y_{min}=x+3,$ then we have
 $F(x,y,z)\geq F(x,y,y)\geq F(x,x+3,x+3)=27x+54=27(z-1)>10z-1,\enskip z\geq3.$
\end{proof}
\section{Results $(i), (ii)$}
\begin{proposition}\label{prop3}
$1)$ For every positive $n,$ except for $n\equiv\pm3 \pmod9,$ $\nu(n)\geq1;$
$2)$ If  $n\equiv\pm3 \pmod9,$ then $\nu(n)=0.$
\end{proposition}
\begin{proof}
1) The statement follows from the following three equalities:
\begin{equation}\label{3}
F(k-1,k,k)=3k-1;
\end{equation}
\begin{equation}\label{4}
F(k-1,k-1,k)=3k-2;
\end{equation}
\begin{equation}\label{5}
F(k,k+1,k+2)=9(k+1).
\end{equation}
\newpage

2) Let, for $n\equiv\pm 3 \pmod9,$ we have
\begin{equation}\label{6}
n= F(x,y,z)).
\end{equation}
However, we show that, if $F(x,y,z)$ is divisible by 3, then it divisible by 9.
Note that, since $x^3\equiv x \mod3,$ then
\begin{equation}\label{7}
 F(x,y,z)\equiv x+y+z \pmod3.
\end{equation}
So, by (\ref{6})
\begin{equation}\label{8}
x+y+z\equiv0 \pmod3.
\end{equation}
 By the symmetry, it is sufficient to
consider the cases $(x,y,z)\equiv(i,i,i)\pmod3, \enskip
 i=0,1,2,$ and $(x,y,z)\equiv(0,1,2) \pmod3.$
Furthermore, note that
\begin{equation}\label{9}
 F(x,y,z)=(x+y+z)(x^2+y^2+z^2-xy-xz-yz)
\end{equation}
and it is easy to see that in the considered cases also
\begin{equation}\label{10}
 x^2+y^2+z^2-xy-xz-yz\equiv0 \pmod3.
 \end{equation}
 So, by (\ref{8}) - (\ref{10}), $F(x,y,z)\equiv0\pmod9$ which contradicts the
 representation (\ref{6}).
\end{proof}
\section{Result $(iii)$}
\begin{proposition}\label{prop4}
For every prime $p\neq3,$ $\nu(p)=\nu(2p)=1.$
\end{proposition}
\begin{proof}
In view of (\ref{3})-(\ref{4}), for every prime $p$ other than 3, we have
 $\nu(p)\geq1.$ However, in (\ref{3})-(\ref{4}) are used the only two possibilities,
 when $z=x+1.$ In both these cases
 \begin{equation}\label{11}
 x^2+y^2+z^2-xy-xz-yz=1.
 \end{equation}
Let us show that, if $z\geq x+2,$ a representation of prime $p$ is
 impossible. In this case $x+y+z\geq2.$ In view of (\ref{9}), if $p=F(x,y,z),$ then
 it should be $x+y+z=p$ such that (\ref{11}) holds. However, using Proposition \ref{prop2},
 we have
 $$x^2+y^2+z^2-xy-xz-yz=\frac {F(x,y,z)}{x+y+z}\geq$$
 \begin{equation}\label{12}
  \frac {9z-10} {(z-2)+2z} \geq 2, \enskip z\geq2,
  \end{equation}
  and (\ref{11}) is impossible. So, for $p\neq3,\enskip \nu(p)=1.$ Finally, for the
   representation of $2p$ in case
  $z\geq x+2,$ note that, since (\ref{11}) does not hold, it should be $x+y+z=p$
  and $x^2+y^2+z^2-xy-xz-yz=2.$ But, according to (\ref{12}), it is possible only
  if $z=2.$ In this case $x=0, \enskip y=0, 1 \enskip or\enskip 2$ and $F(x,y,z)=8,
   9\enskip or \enskip16.$ Thus, for $p\neq3,\enskip \nu(2p)=1.$
  \end{proof}
  \newpage
  For example, we have a unique representation
  $$ p=x^3+y^3+z^3-3xyz$$
  with $x=y=z-1=\frac{p-1}{3}$ if prime $p\equiv1 \pmod3$ and with
  $x+1=y=z=\frac{p+1}{3}$ if prime $p\equiv2 \pmod3.$\newline
  Also we have a unique representation
  $$ 2p=x^3+y^3+z^3-3xyz$$
  with $x+1=y=z=\frac{2p+1}{3}$ if prime $p\equiv1 \pmod3$ and with 
  $x=y=z-1=\frac{2p-1}{3}$ if prime $p\equiv2 \pmod3.$ 
  
 \section{Result (iv)}
\begin{lemma}\label{L1}
\begin{equation}\label{13}
 (F(x,y,z))^3=F(u,v,w),
\end{equation}
 where
  $$u=F(x,y,z)+9xyz,\enskip v=3(x^2y+y^2z+z^2x),\enskip w=3(x^2z+z^2y+y^2x).$$
  \end{lemma}
  \begin{proof}
 The identity is proved straightforward.
  \end{proof}
 \begin{lemma}\label{L2}
If the numbers $x,y,z$ in $(\ref{13})$ form an arithmetic progression with the
difference $d\geq1,$ then the numbers $v,u,w$ form an arithmetic progression with the
difference $3d^3.$
\end{lemma}
\begin{proof}
Let for $x\geq0, d\geq1,$ we have $y=x+d, \enskip z=x+2d.$ Then
$$v=3(x^2y+y^2z+z^2x)=9x^3+27x^2d+27xd^2+6d^3,$$
$$u=x^3+y^3+z^3+6xyz=9x^3+27x^2d+27xd^2+9d^3,$$
$$w=3(x^2z+z^2y+y^2x)=9x^3+27x^2d+27xd^2+12d^3.$$
Thus $u=v+d_1, \enskip w=v+2d_1,$ where $d_1=3d^3.$
\end{proof}
\begin{remark}\label{r1}
Since here $v<u<w,$ then $(\ref{13})$ we can write in the form $(F(x,y,z))^3=
F(v,u,w);$ further $(F(v,u,w))^3=F(\xi,\eta,\zeta),$ such that $\xi<\eta<\zeta,$ etc.
\end{remark}

\begin{proposition}\label{prop5}
$\limsup (\nu(n))=\infty.$
\end{proposition}
\begin{proof}
Consider sequence $27,27^3,27^{3^2},27^{3^3},...,27^{3^k},... .$
Representation $27^{3^k}=F(0,0,27^{3^{k-1}})$ we call trivial. We are interested in
non-trivial representations of $b_k=27^{3^k}.$ Note that $b_0=27$ has a unique
non-trivial representation defined by (\ref{5}): $b_0=F(2,3,4).$ Thus, by Lemma
\ref{L2}, $b_1=b_0^3$ has at least 2 distinct non-trivial representations: by
(\ref{5}) with $d_1=1$ and with $d_2=3.$
\newpage
 Further, again by Lemma \ref{L2},
 $b_2=b_1^3$ has at least 3 distinct non-trivial representations: by (\ref{5}) with
  $d_1=1,$ $d_2=3$ and $3\cdot3^3=3^4.$
Analogously $b_3=b_2^3$ has at least 4 distinct non-trivial representations: by (\ref{5}) with
  $d_1=1,$ $d_2=3,$ $d_3=3^4$ and $d_4=3\cdot81^3=3^{13};$ ..., $b_k=b_{k-1}^3$ has
at least $k+1$ distinct non-trivial representations: $1,3,3^4,3^{13},...,
3^{(3^k-1)/2}.$ This completes the proof.
\end{proof}
We give also the second proof.
\begin{proof}
We use the homogeneity of $F(x,y,z)$ of degree 3. By induction, show that $\nu(8^k)
\geq k+1.$ It is evident for $k=0.$ Suppose that it is true for some
value of $k.$
Take $k+1$ triples $(x_i,y_i,z_i)$ such that $8^k = F(x_i, y_i, z_i),\enskip
 i=1,...,k+1.$ Then for $k+1$ triples of even numbers $(2x_i,2y_i,2z_i),$
we have $8^{k+1} = F(2x_i,2y_i,2z_i).$ But, by (\ref{3})-(\ref{4}), always there
is a triple of not all even numbers $x=(n-1)/3,\enskip y=(n-1)/3,\enskip z=(n+2)/3)$
or $x=((n-2)/3,\enskip y=(n+1)/3,\enskip z=(n+1)/3),$ where $n=8^{k+1},$ for which
$8^{k+1} = F(x,y,z).$ So $\nu(8^{k+1})\geq k+2.$
\end{proof}
\section{Result (v)}
\begin{lemma}\label{L3}
There is a unique representation of $8^k$ by the form $F(x,y,z)$ with not all even
numbers $x,y,z.$
\end{lemma}
\begin{proof}
 In (\ref{3})-(\ref{4}) we used the only two possibilities,
 when $z=x+1$ and in both these cases we have the equality (\ref{11}). This gives
 one representation of $8^k$, when $8^k\equiv1\pmod3$ (even $k$) and one
representation of $8^k$, when $8^k\equiv2\pmod3.$(odd $k$). Let now $z\geq x+2.$
Since $Fx,y,z)=(x+y+z)(x^2+y^2+z^2-xy-xz-yz),$ and, in view of the symmetry, the
case, when the numbers  $x,y,z$ in a representation of $8^k$ are not all even,
reduces, say, for the case when $x$ and $y$ are odd, while $z$ is even. But then
$x^2+y^2+z^2-xy-xz-yz$ is odd. In Section 3 we saw that, in the condition
$z\geq x+2,\enskip x^2+y^2+z^2-xy-xz-yz\geq2.$  So it is odd $\geq3.$ This is impossible
in representation of $8^k$ by $F(x,y,z).$
\end{proof}
\begin{theorem}\label{t1}
For every positive $n,$ there exists $k$ such that $\nu(k)=n.$
\end{theorem}
\begin{proof}
In the second proof of Result $iv,$ we showed that $\nu(8^k)\geq k+1.$ To prove the
theorem, it suffices to prove that really we have here the equality: $\nu(8^k)=k+1.$
Again use induction. Suppose that it is true for some value of $k.$ As in the
second proof of Result $iv,$ take $k+1$ triples $(x_i,y_i,z_i)$ such that
$8^k = F(x_i, y_i, z_i),\enskip i=1,...,k+1.$ By Lemma \ref{L3}, among these triples
there exists a unique triple, say, $(x_{k+1}, y_{k+1}, z_{k+1})$ with not all 
even numbers.
\newpage
Then for $k+1$ triples of even numbers $(2x_i,2y_i,2z_i),$ we have $8^{k+1} =
 F(2x_i,2y_i,2z_i),$ and only one of them $(2x_{k+1}, 2y_{k+1}, 2z_{k+1})$ contains
 not all numbers divisible by 4. Besides, there is a unique triple with odd two numbers.
 Suppose now, that there is an additional $(k+3)$-th triple $(x*,y*,z*)$  such that
 $8^{k+1}=F(x*,y*,z*).$ All numbers $x*,y*,z*$ should be divisible by 4. But then
 a triple $(x*/2,y*/2,z*/2)$ is an additional $(k+2)$-th triple for representation
 of $8^k.$ This contradicts the inductional supposition. The theorem follows.
\end{proof}
In conclusion, note that the sequence of $\{\nu(n)\}$ is A261029 \cite{3} (including
also $n=0).$ Besides, the smallest numbers $k=k(n)$ from Theorem 1 are
presented in our with Peter J. C. Moses sequence A260935 \cite{3}.
\section{On a Carmichael paper}
While browsing the Bulletin of the American Mathematical Society, Michel Marcus
found a Carmichael paper \cite{1} on the same topic (now it is available in the
sequence A074232). The methods of \cite{1} and the present paper are quite different.
So comparing the results, we can consider proof of (i)-(iii) as a \slshape
short proof \upshape\enskip of the main results of \cite{1}, while (iv)-(v) give new results.\newline
\indent The author is happy to unwittingly continue with a new approach a
research of the outstanding mathematician Robert Daniel Carmichael in exact CENTENARY
(Aug 1915 - Aug 2015) of his paper.\newline
\indent Note that we published almost at the same time also the paper
\cite{2} which was inspired by the sequences A072670 and A260803 \cite{3} by R.
Zumkeller and D. A. Corneth respectively. These sequences with its restriction
conditions essentially inspired also the present paper, since the author always
remembered the remarkable form $ x^3+y^3+z^3-3xyz$ which is the determinant of
the circulant matrix with the first row $ (x,y,z).$

\end{document}